\documentclass[12pt]{amsart}
    \usepackage[OT2,T1]{fontenc}
    \usepackage{amsmath,amsfonts,amsthm,amssymb,mathrsfs,
    amsxtra,amscd,latexsym,color,mathdots,verbatim}
    \usepackage[hmargin=2cm,vmargin=3cm]{geometry}
    \DeclareSymbolFont{cyrletters}{OT2}{wncyr}{m}{n}
    \DeclareMathSymbol{\Sha}{\mathalpha}{cyrletters}{"58}

     \newtheorem{thm}{Theorem}[section]
     
     \newtheorem{cor}[thm]{Corollary}
     
     \newtheorem{lem}[thm]{Lemma}

     \newtheorem{dfn}[thm]{Definition}

     \newtheorem{rmk}[thm]{Remark}
     
     \theoremstyle{definition}
     \theoremstyle{remark}
     \numberwithin{equation}{section}

    \newcommand{\sm}{\left(\begin{smallmatrix}}
    \newcommand{\esm}{\end{smallmatrix}\right)}
    \newcommand{\mat}{\left(\begin{matrix}}
    \newcommand{\emat}{\end{matrix}\right)}

    \newcommand{\mbf}{\mathbf}

    \def\CC{\mathbb{C}}
    \def\HH{\mathbb{H}}

    \def\GL{\mathrm{GL}}
    \def\Mp{\mathrm{Mp}}
    \def\SL{\mathrm{SL}}

\begin{document}

    \title{Results on the Non-Vanishing of Derivatives of $L$-Functions of Vector-Valued Modular Forms}

   \author{Subong Lim}
   \address{Department of Mathematics Education, Sungkyunkwan University, Jongno-gu, Seoul 110-745, Republic of Korea}
   \email{subong@skku.edu}
   \author{Wissam Raji}
   \address{Department of Mathematics, American University of Beirut (AUB) and the Number Theory Unit at the Center for Advanced Mathematical Sciences (CAMS) at AUB, Beirut, Lebanon}
   \email{wr07@aub.edu.lb}
\date{}

\maketitle
   
   \begin{abstract}
We show a non-vanishing result for the averages of the derivatives of $L$-functions associated with the orthogonal basis of the space of vector-valued cusp forms of weight $k\in \frac12 \mathbb{Z}$ on the full group in the critical strip. We also show the existence of at least one basis element whose $L$-function does not vanish under certain conditions.
As an application, we generalize our result to Kohnen's plus space and prove an analogous result for Jacobi forms. 
\end{abstract} 
  \section{Introduction}

\par The theory of $L$-functions play a crucial role in both number theory and arithmetic geometry. $L$-functions exhibit natural connections with various mathematical subjects including number fields, automorphic forms, Artin representations, Shimura varieties, abelian varieties, and intersection theory. The central values of $L$-functions and their derivatives reveal important connections to the geometric and arithmetic properties of Shimura varieties such as the Gross-Zagier formula, Colmez's conjecture, and the averaged Colmez formula. On the other hand, vector-valued modular forms are important generalizations of elliptic modular forms that arise naturally in the theory of Jacobi forms, Siegel modular forms, and Moonshine. Being an important tool to tackle classical problems in the theory of modular forms, Selberg used these forms to give an estimation for the Fourier coefficients of the classical modular forms \cite{S}. Borcherds in \cite{B1} and \cite{B2} used vector-valued modular forms associated with Weil representations to provide a description of the Fourier expansion of various theta liftings. Some applications of vector-valued modular forms stand out in high-energy physics by mainly providing a method of differential equations in order to construct the modular multiplets, and also revealing the simple structure of the modular invariant mass models \cite{DL}. Other applications concerning vector-valued modular forms of half-integer weight seem to provide a simple solution to the Riemann-Hilbert problem for representations of the modular group \cite{BG}. So it is only natural to study the $L$-functions of vector-valued modular forms and their properties as a buildup that aligns with the development of a Hecke theory to the space of vector-valued modular forms.

  In \cite{LR}, we show that averages of $L$-functions associated with vector-valued cusp  forms do not vanish when the average is taken over the orthogonal basis of the space of vector-valued cusp forms.  To illustrate, we let $\{f_{k,1}, \ldots, f_{k,d_k}\}$ be an orthogonal basis of the space $S_{k,\chi,\rho}$ of vector-valued cusp forms  with Fourier coefficients $b_{k,l,j}(n)$, where $\chi$ is a multiplier system of weight $k\in\frac12 \mathbb{Z}$ on $\mathrm{SL}_2(\mathbb{Z})$ and
$\rho:\SL_2(\mathbb{Z})\to \GL_m(\CC)$ is an $m$-dimensional unitary complex representation.
We also let $t_0\in\mathbb{R}, \epsilon>0$, and $1\leq i\leq m$. 
  Then, there exists a constant $C(t_0, \epsilon, i)>0$ such that for $k>C(t_0, \epsilon, i)$ the function 
  \[
  \sum_{l=1}^{d_k} \frac{<L^*(f_{k,l}, s),\mbf{e}_i>}{(f_{k,l}, f_{k,l})} b_{k,l,i}(n_{i,0})
  \]
  does not vanish at any point $s=\sigma + it_0$ with $\frac{k-1}{2}<\sigma < \frac k2 - \epsilon$, where $<L^*(f_{k,l}, s),\mbf{e}_i>$ denotes the $i$th component of $L^*(f_{k,l}, s)$. 

\par Kohnen, Sengupta, and Weigel in \cite{KSW} proved a nonvanishing result for the derivatives of $L$-functions in the critical strip for elliptic modular forms on the full group.  In \cite{Raj}, the second author generalized their result to modular forms of half-integer weight on the plus space.  In this paper, we show analogous results for the averages of the derivatives of $L$-functions for the orthogonal basis of the space of vector-valued cusp forms in the critical strip.  In particular, given $k\in\frac12 \mathbb{Z}$, $\chi$ a multiplier system of weight $k$ on $\mathrm{SL}_2(\mathbb{Z})$, $t_0\in\mathbb{R}, \epsilon>0, 1\leq i\leq m$, and $n$ a positive integer, we show that
there exists a constant $C(t_0, \epsilon,i, n)>0$ such that for $k>C(t_0, \epsilon, i, n)$ the function
\begin{equation*}
\sum_{l=1}^{g_k} \frac{b_{k,l,i}(n_{i,0})}{(f_{k,l}, f_{k,l})}  \frac{d^n}{ds^n} <L^*(f_{k,l},s), \mbf{e}_i>
\end{equation*}
does not vanish at any point $s = \sigma + it$ with $t = t_0, \frac{k-1}{2} <\sigma < \frac k2 -\epsilon$. 

\par The isomorphism between the space of Jacobi forms of weight $k$ and  index $m$ on $SL_2(\mathbb{Z})$ and the space of vector-valued modular cusp forms with a specific multiplier system and a given Weil representation depending on $m$ leads to an analogous result for Jacobi forms.   We also give a similar result for cusp forms in the plus space.

\section{The Kernel Function}
 In this section, we define the kernel function $R_{k,s,i}$ and determine its Fourier expansion.  The kernel function being a cusp form will play an important role in determining the coefficients of a given cusp form in terms of $L$-functions when the given cusp form is written in terms of the orthogonal basis. So, let $\Gamma = \mathrm{SL}_2(\mathbb{Z})$, $k\in \frac12 \mathbb{Z}$ and $\chi$ a unitary multiplier system of weight $k$ on $\Gamma$, i.e. $\chi:\Gamma\to\mathbb{C}$ satisfies the following conditions:
  \begin{enumerate}
  \item $|\chi(\gamma)| = 1$ for all $\gamma\in \Gamma$.
  \item $\chi$ satisfies the consistency condition
  \[
  \chi(\gamma_3) (c_3\tau + d_3)^k = \chi(\gamma_1)\chi(\gamma_2) (c_1\gamma_2\tau + d_1)^k (c_2\tau+d_2)^k,
  \]
  where $\gamma_3 = \gamma_1\gamma_2$ and $\gamma_i = \sm a_i&b_i\\c_i&d_i\esm\in \Gamma$ for $i=1,2$, and $3$.
  \end{enumerate}
  Let $m$ be a positive integer and $\rho:\Gamma\to \mathrm{GL}(m, \mathbb{C})$ a $m$-dimensional unitary complex representation. 
  Let $\{\mbf{e}_1, \ldots, \mbf{e}_m\}$ denote the standard basis of $\mathbb{C}^m$. 
  For a vector-valued function $f = \sum_{j=1}^m f_j \mbf{e}_j$ on $\mathbb{H}$ and $\gamma\in \Gamma$, define a slash operator by
  \[
  (f|_{k,\chi,\rho}\gamma)(\tau):= (c\tau+d)^{-k} \chi^{-1}(\gamma)\rho^{-1}(\gamma) f(\gamma \tau).
  \]
  The definition of the vector-valued modular forms is given as follows.
  
  \begin{dfn}
  A vector-valued modular form of weight $k\in\frac12\mathbb{Z}$, multiplier system $\chi$, and type $\rho$ on $\Gamma$ is a sum $f = \sum_{j=1}^m f_j \mbf{e}_j$ of functions holomorphic in  $\mathbb{H}$ satisfying the following conditions:
  \begin{enumerate}
  \item $f|_{k,\chi,\rho}\gamma = f$ for all $\gamma\in \Gamma$.
  \item For each $1\leq j\leq m$, each function $f_j$ has a Fourier expansion of the form
  \[
  f_i(\tau) = \sum_{n+\kappa_j\geq 0} a_j(n)e^{2\pi i(n+\kappa_j)\tau}.
  \]
  Here and throughout the paper, $\kappa_j$ is a certain positive number with $0\leq \kappa_j<1$. 
  \end{enumerate}
  \end{dfn}
  
  The space of all vector-valued modular forms of weight $k$, multiplier system $\chi$, and type $\rho$ on $\Gamma$ is denoted by $M_{k,\chi,\rho}$. 
  There is a subspace $S_{k,\chi,\rho}$ of vector-valued cusp forms for which we require that each $a_j(n) = 0$ when $n+\kappa_j$ is non-positive.
  
  For a vector-valued cusp form $f(\tau) = \sum_{j=1}^m \sum_{n+\kappa_j>0} a_j(n)e^{2\pi i(n+\kappa_j)\tau}\mbf{e}_j\in S_{k,\chi,\rho}$, we see that $a_j(n) = O(n^{k/2})$ for every $1\leq j\leq m$ as $n\to\infty$ by the same argument for elliptic modular forms. 
  Then, the vector-valued $L$-function defined by
  \[
  L(f,s) := \sum_{j=1}^m \sum_{n+\kappa_j>0} \frac{a_j(n)}{(n+\kappa_j)^s}\mbf{e}_j
  \]
  converges absolutely for $\mathrm{Re}(s)\gg0$. 
   This has an integral representation
  \[
  \frac{\Gamma(s)}{(2\pi)^s} L(f,s) = \int_0^\infty f(iv) v^s \frac{dv}{v}.
  \]  
  From this, we see that it has an analytic continuation to $\mathbb{C}$ and a functional equation given by
 \[
 L^*(f,s) = i^k \chi(S) \rho(S) L^*(f,k-s),
 \]
 where $L^*(f,s) = \frac{\Gamma(s)}{(2\pi)^s} L(f,s)$ and $S = \sm 0&-1\\1&0\esm$.
 
 Let $i$ be an integer with $1\leq i\leq m$. 
 Define 
 \[
 p_{s,i}(\tau) := \tau^{-s} \mbf{e}_i.
 \]
 For $s\in\mathbb{C}$ with $1<\mathrm{Re}(s)<k-1$, we define the kernel function by
 \[
 R_{k,s,i} := \gamma_k(s) \sum_{\gamma\in \Gamma} p_{s,i} |_{k,\chi,\rho}\gamma,
 \]
 where $\gamma_{k}(s) := \frac12 e^{\pi is/2} \Gamma(s) \Gamma(k-s)$.
 Then, this series converges absolutely uniformly whenever $\tau = u+iv$ satisfies $v\geq \epsilon, v\leq 1/\epsilon$ for a given $\epsilon>0$ and $s$ varies over a compact set.
 Moreover, it is a vector-valued cusp form in $S_{k,\chi,\rho}$. 
 
 We write $<\cdot, \cdot>$ for the standard scalar product on $\mathbb{C}^m$, i.e.
 \[
 \left< \sum_{j=1}^m \lambda_j \mbf{e}_j, \sum_{j=1}^m \mu_j \mbf{e}_j \right> = \sum_{j=1}^m \lambda_j \overline{\mu_j}.
 \]
 Then, for $f,g\in M_{k,\chi,\rho}$, we define the Petersson scalar product of $f$ and $g$ by
 \[
 (f,g) := \int_{\mathcal{F}} <f(\tau), g(\tau)> v^k \frac{dudv}{v^2}
 \]
 if the integral converges, where $\mathcal{F}$ is the standard fundamental domain for the action of $\Gamma$ on $\mathbb{H}$.
 Then,  by  \cite[Lemma 3.1]{LR}, we have
 \begin{equation} \label{Petersson}
 (f, R_{k,\bar{s}, i}) = c_k <L^*(f, s), \mbf{e}_i >,
 \end{equation}
 where $c_k := \frac{(-1)^{k/2} \pi (k-2)!}{2^{k-2}}$.

 We can also compute the Fourier expansion of $R_{k,s,i}$.

\begin{lem} \label{Fourier} \cite[Lemma 3.2]{LR}
The function $R_{k,s,i}$ has the Fourier expansion
\[
R_{k,s,i}(\tau) = \sum_{j=1}^m \sum_{n+\kappa_j>0} r_{k,s,i,j}(n)e^{2\pi i(n+\kappa_j)\tau},
\]
where $r_{k,s,i,j}(n)$ is given by
\begin{eqnarray*}
r_{k,s,i,j}(n) &=& \delta_{i,j}  (2\pi)^s \Gamma(k-s)(n+\kappa_i)^{s-1}\\
&& + \chi^{-1} \left( \sm 0&-1\\1&0\esm \right)\rho^{-1}\left( \sm 0&-1\\1&0\esm \right)_{j,i} (-1)^{k/2}(2\pi)^{k-s}\Gamma(s) (n+\kappa_j)^{k-s-1}\\
&&+\frac{ (-1)^{k/2}}{2} (2\pi)^k (n+\kappa_j)^{k-1} \frac{\Gamma(s)\Gamma(k-s)}{\Gamma(k)}\sum_{\substack{(c,d)\in\mathbb{Z}^2 \\ (c,d)=1, ac>0}}    c^{-k} \left( \frac ca \right)^{s}\\
&&\quad\times \bigg(e^{2\pi i(n+\kappa_j)d/c} e^{\pi is}\chi^{-1}\left( \sm a&b\\c&d\esm\right)  \rho^{-1}\left( \sm a&b\\c&d\esm\right)_{j,i}\ _1F_1(s, k; -2\pi in/(ac))\\
&&\quad \quad + e^{-2\pi i(n+\kappa_j)d/c} e^{-\pi is} \chi^{-1}\left( \sm -a&b\\c&-d\esm\right) \rho^{-1}\left( \sm -a&b\\c&-d\esm\right)_{j,i}\ _1F_1(s, k; 2\pi in/(ac)) \bigg),
\end{eqnarray*}
where $_1 F_1(\alpha, \beta;z)$ is Kummer's degenerate hypergeometric function.
\end{lem}

\section{The Main Theorem}
In this section, we give the main theorem for the non-vanishing of an average of derivatives of $L$-functions in the critical strip.  We also derive a corollary about the existence of at least one $L$-function whose derivative does not vanish.  To do so, let 
    \begin{equation*}
  n_{i,0} := 
  \begin{cases}
  1 & \text{if $\kappa_i = 0$},\\
  0 & \text{if $\kappa_i\neq0$}.
  \end{cases}
  \end{equation*}
Then, we have the following theorem. 

\begin{thm} \label{main}
Let $k\in\frac12 \mathbb{Z}$ and let $\chi$ be a multiplier system of weight $k$ on $\mathrm{SL}_2(\mathbb{Z})$.
Suppose that $\{f_{k,1}, \ldots, f_{k,g_k}\}$ is an orthogonal basis of $S_{k,\chi,\rho}$ with Fourier expansions
\[
 f_{k,l}(\tau) = \sum_{j=1}^m \sum_{n+\kappa_j>0} b_{k,l,j}(n)e^{2\pi i(n+\kappa_j)\tau}\ (1\leq l\leq g_k),
 \]
where $g_k := \dim S_{k,\chi,\rho}$.
Let $t_0\in\mathbb{R}, \epsilon>0, 1\leq i\leq m$, and $n$ a positive integer. 
Then, there exists a constant $C(t_0, \epsilon, i,n)>0$ such that for $k>C(t_0, \epsilon,i, n)$ the function
\begin{equation} \label{weightedsum}
\sum_{l=1}^{g_k} \frac{b_{k,l,i}(n_{i,0})}{(f_{k,l}, f_{k,l})} \frac{d^n}{ds^n} <L^*(f_{k,l},s), \mbf{e}_i>
\end{equation}
does not vanish at any point $s = \sigma + it$ with $t = t_0, \frac{k-1}{2} <\sigma < \frac k2 -\epsilon$.
\end{thm}

\begin{proof}
We follow the argument in the proof of Theorem 3.1 in \cite{KSW}.
For each $1\leq i\leq m$, by (\ref{Petersson}), we have 
\begin{equation} \label{Rexpansion}
R_{k,s,i} = c_k \sum_{l=1}^{g_k} \frac{<L^*(f_{k,l}, s),\mbf{e}_i>}{(f_{k,l}, f_{k,l})} f_{k,l}.
\end{equation}
If we take the first Fourier coefficients of $i$th component function on both sides of (\ref{Rexpansion}), then by Lemma \ref{Fourier} we have
\begin{eqnarray} \label{firstterm}
&&   (2\pi)^s \Gamma(k-s)(n_{i,0}+\kappa_i)^{s-1}\\
\nonumber && + \chi^{-1} \left( \sm 0&-1\\1&0\esm \right)\rho^{-1}\left( \sm 0&-1\\1&0\esm \right)_{i,i} (-1)^{k/2}(2\pi)^{k-s}\Gamma(s) (n_{i,0}+\kappa_i)^{k-s-1}\\
\nonumber &&+\frac{ (-1)^{k/2}}{2} (2\pi)^k (n_{i,0}+\kappa_i)^{k-1}\sum_{\substack{(c,d)\in\mathbb{Z}^2 \\ (c,d)=1, ac>0}}    c^{-k} \left( \frac ca \right)^{s}\\
\nonumber &&\quad\times \bigg(e^{2\pi i(n_{i,0}+\kappa_j)d/c} e^{\pi is}\chi^{-1}\left( \sm a&b\\c&d\esm\right)  \rho^{-1}\left( \sm a&b\\c&d\esm\right)_{i,i}\ _1f_1(s, k; -2\pi in_{i,0}/(ac))\\
\nonumber &&\quad \quad + e^{-2\pi i(n_{i,0}+\kappa_i)d/c} e^{-\pi is} \chi^{-1}\left( \sm -a&b\\c&-d\esm\right) \rho^{-1}\left( \sm -a&b\\c&-d\esm\right)_{i,i}\ _1f_1(s, k; 2\pi in_{i,0}/(ac)) \bigg)\\
\nonumber &&=  c_k \sum_{l=1}^{g_k} \frac{<L^*(f_{k,l}, s),\mbf{e}_i>}{(f_{k,l}, f_{k,l})} b_{k,l,i}(n_{i,0}),
\end{eqnarray}
where  
\[
_1 f_1(\alpha, \beta;z) := \frac{\Gamma(\alpha) \Gamma(\beta-\alpha)}{\Gamma(\beta)}\ _1F_1(\alpha, \beta;z).
\]

We assume that (\ref{weightedsum}) is zero. 
If we take the $n$th derivative with respect to $s$ on both sides in (\ref{firstterm}), then we have
\begin{eqnarray} \label{derivative}
&&\frac{d^n}{d s^n} \bigg[ (2\pi)^s \Gamma(k-s)(n_{i,0}+\kappa_i)^{s-1}\bigg]\\
\nonumber && = -\frac{d^n}{ds^n}\bigg[  \chi^{-1} \left( \sm 0&-1\\1&0\esm \right)\rho^{-1}\left( \sm 0&-1\\1&0\esm \right)_{i,i} (-1)^{k/2}(2\pi)^{k-s}\Gamma(s) (n_{i,0}+\kappa_i)^{k-s-1}\bigg]\\
\nonumber &&-\frac{d^n}{ds^n}\bigg[ \frac{ (-1)^{k/2}}{2} (2\pi)^k (n_{i,0}+\kappa_i)^{k-1}\sum_{\substack{(c,d)\in\mathbb{Z}^2 \\ (c,d)=1, ac>0}}    c^{-k} \left( \frac ca \right)^{s}\\
\nonumber &&\quad\times \bigg(e^{2\pi i(n_{i,0}+\kappa_j)d/c} e^{\pi is}\chi^{-1}\left( \sm a&b\\c&d\esm\right)  \rho^{-1}\left( \sm a&b\\c&d\esm\right)_{i,i}\ _1f_1(s, k; -2\pi in_{i,0}/(ac))\\
\nonumber &&\quad \quad + e^{-2\pi i(n_{i,0}+\kappa_i)d/c} e^{-\pi is} \chi^{-1}\left( \sm -a&b\\c&-d\esm\right) \rho^{-1}\left( \sm -a&b\\c&-d\esm\right)_{i,i}\ _1f_1(s, k; 2\pi in_{i,0}/(ac)) \bigg) \bigg].
\end{eqnarray}

Then, the left-hand side of (\ref{derivative}) is equal to 
\begin{eqnarray*}
&& \frac{1}{n_{i,0}+\kappa_i} \sum_{\nu =0}^n \frac{d^\nu}{ds^\nu} [(2\pi(n_{i,0}+\kappa_i))^s] \frac{d^{n-\nu}}{ds^{n-\nu}} \Gamma(k-s)\\
&&= \frac{(2\pi(n_{i,0}+\kappa_i))^s}{n_{i,0}+\kappa_i}  \sum_{\nu =0}^n (-1)^{n-\nu} \binom{n}{\nu} (\log(2\pi(n_{i,0}+\kappa_i)))^{\nu} \Gamma^{(n-\nu)}(k-s)\\
&&= (2\pi)^s (n_{i,0}+\kappa_i)^{s-1}  (\log(2\pi(n_{i,0}+\kappa_i)))^n \Gamma(k-s)\\
&& \qquad +  (2\pi)^s (n_{i,0}+\kappa_i)^{s-1} \sum_{\nu =0}^{n-1}(-1)^{n-\nu} \binom{n}{\nu} (\log(2\pi(n_{i,0}+\kappa_i)))^{\nu} \Gamma^{(n-\nu)}(k-s).
\end{eqnarray*}
Then, we have
\begin{eqnarray*}
&&\frac{1}{(2\pi)^s (n_{i,0}+\kappa_i)^{s-1} \Gamma(k-s)} \cdot  \frac{d^n}{d s^n} \bigg[ (2\pi)^s \Gamma(k-s)(n_{i,0}+\kappa_i)^{s-1}\bigg]\\
&&= (\log(2\pi(n_{i,0}+\kappa_i)))^n  +  \sum_{\nu =0}^{n-1}(-1)^{n-\nu} \binom{n}{\nu} (\log(2\pi(n_{i,0}+\kappa_i)))^{\nu} \frac{ \Gamma^{(n-\nu)}(k-s)}{\Gamma(k-s)}.
\end{eqnarray*}

 Let $\psi(s) := \frac{\Gamma'(s)} {\Gamma(s)}$. 
 Then, one can see that $\frac{\Gamma^{(n)}(s)}{\Gamma(s)}$ is a polynomial $P(\psi, \psi^{(1)}, \ldots, \psi^{(n-1)})$ with integral coefficients and it contains the term $\psi^n$ which is the highest power of $\psi$ occurring in $P$. 
 It is known that $\psi$ satisfies the following asymptotic formulas
 \begin{equation*} \label{psi1}
 \psi(s) \sim \log(s) - \frac{1}{2s} - \sum_{\nu=1}^\infty \frac{B_{2\nu}} {2\nu s^{2\nu}}
 \end{equation*}
 and
 \begin{equation*} \label{psi2}
 \psi^{(n)}(s) \sim (-1)^{n-1} \left( \frac{(n-1)!}{s^n} + \frac{n!}{2s^{n+1}} + \sum_{\nu=0}^\infty B_{2\nu} \frac{(2\nu+n-1)!}{(2\nu)! s^{2\nu+n}}\right)
 \end{equation*}
 for $s\to\infty$ in $|\arg(s)|<\pi$, where $B_n$ denotes the $n$th Bernoulli number (for example, see \cite[6.3.18 and 6.4.11]{AS}). 
 Let $s = \frac{k}{2} - \delta + it_0\ (\epsilon < \delta < \frac12)$.
 Then the leading term of $\frac{\Gamma^{(n-\nu)}(k-s)} {\Gamma(k-s)}$ for $0\leq \nu\leq n-1$ is $(\log(\frac k2 +\delta - it_0))^{n-\nu}$ as $k\to\infty$ and $\psi^{(n)}(s) = o\left( \frac{1}{|s|^n}\right)$ as $|s|\to\infty$ in $|\arg(s)|<\pi$ for $n\in\mathbb{N}$.
 Therefore, we have
 \begin{equation*}
  \sum_{\nu =0}^{n-1}(-1)^{n-\nu} \binom{n}{\nu} (\log(2\pi(n_{i,0}+\kappa_i)))^{\nu} \frac{ \Gamma^{(n-\nu)}(k-s)}{\Gamma(k-s)} = Q\left( \log(\frac k2 +\delta - it_0)\right) + o(1)
 \end{equation*}
 as $k\to\infty$, where $Q$ is a polynomial of degree $n$ and its highest coefficient is $(-1)^n$.
 
 For the first term on the right-hand side of (\ref{derivative}), we have
 \begin{eqnarray*}
&& \frac{d^n}{ds^n}\bigg[  \chi^{-1} \left( \sm 0&-1\\1&0\esm \right)\rho^{-1}\left( \sm 0&-1\\1&0\esm \right)_{i,i} (-1)^{k/2}(2\pi)^{k-s}\Gamma(s) (n_{i,0}+\kappa_i)^{k-s-1}\bigg] \\
&& = \chi^{-1} \left( \sm 0&-1\\1&0\esm \right)\rho^{-1}\left( \sm 0&-1\\1&0\esm \right)_{i,i} (-1)^{k/2}\frac{(2\pi(n_{i,0}+\kappa_i))^{k}}{n_{i,0}+\kappa_i}
\sum_{\nu = 0}^n \binom{n}{\nu} \frac{d^\nu}{ds^{\nu}} [ (2\pi  (n_{i,0}+\kappa_i))^{-s}] \frac{d^{n-\nu}}{ds^{n-\nu}} \Gamma(s)\\
&& = \chi^{-1} \left( \sm 0&-1\\1&0\esm \right)\rho^{-1}\left( \sm 0&-1\\1&0\esm \right)_{i,i} (-1)^{k/2}\frac{(2\pi(n_{i,0}+\kappa_i))^{k-s}}{n_{i,0}+\kappa_i}\\
&&\qquad \times
\sum_{\nu = 0}^n  (-1)^{\nu} \binom{n}{\nu} \log (2\pi  (n_{i,0}+\kappa_i))^\nu  \Gamma^{(n-\nu)}(s).
 \end{eqnarray*}
 If we divide this by $(2\pi)^s  (n_{i,0}+\kappa_i)^{s-1} \Gamma(k-s)$, then we have
 \begin{eqnarray} \label{secondterm}
 &&\chi^{-1} \left( \sm 0&-1\\1&0\esm \right)\rho^{-1}\left( \sm 0&-1\\1&0\esm \right)_{i,i} (-1)^{k/2}\frac{(2\pi(n_{i,0}+\kappa_i))^{k-2s}}{(n_{i,0}+\kappa_i)^2}\\
\nonumber &&\qquad \times
\sum_{\nu = 0}^n  (-1)^{\nu} \binom{n}{\nu} \log (2\pi  (n_{i,0}+\kappa_i))^\nu \frac{ \Gamma^{(n-\nu)}(s)}{\Gamma(s)} \cdot \frac{\Gamma(s)}{\Gamma(k-s)}.
 \end{eqnarray}
  Let $s = \frac{k}{2} - \delta + it_0\ (\epsilon < \delta < \frac12)$.
  Then by \cite[6.1.23 and 6.1.47]{AS},   we have
  \begin{eqnarray*}
 \left| \frac{\Gamma(s)}{\Gamma(k-s)} \right| = \left| \frac k2 + it_0 \right|^{-2\delta} \cdot \left |1+ O\left(\frac{1}{|\frac k2 + it_0|}\right)\right|,
 \end{eqnarray*}
 where the $O$ constant is absolute, uniformly in $\epsilon<\delta<\frac12$.
 On the other hand, the highest order term in $\frac{\Gamma^{(n-\nu)}(s)}{\Gamma(s)}$  is $\left( \psi \left(\frac k2 -\delta + it_0\right)\right)^{n-\nu}$. 
 This behaves like $(\log(\frac k2 - \delta + it_0))^{n-\nu}$ for $0\leq \nu < n$ as $k\to\infty$. 
 Thus, we can see that all terms in the sum in (\ref{secondterm}) go to zero as $k\to\infty$. 
 
 The second term on the right-hand side of (\ref{derivative}) is equal to 
 \begin{eqnarray}  \label{thirdterm}
 \nonumber &&- \frac{ (-1)^{k/2}}{2} (2\pi)^k (n_{i,0}+\kappa_i)^{k-1}\sum_{\substack{(c,d)\in\mathbb{Z}^2 \\ (c,d)=1, ac>0}}    c^{-k} \frac{d^n}{ds^n}\bigg[\left( \frac ca \right)^{s}\\
\nonumber &&\quad\times \bigg(e^{2\pi i(n_{i,0}+\kappa_j)d/c} e^{\pi is}\chi^{-1}\left( \sm a&b\\c&d\esm\right)  \rho^{-1}\left( \sm a&b\\c&d\esm\right)_{i,i}\ _1f_1(s, k; -2\pi in_{i,0}/(ac))\\
\nonumber &&\quad \quad + e^{-2\pi i(n_{i,0}+\kappa_i)d/c} e^{-\pi is} \chi^{-1}\left( \sm -a&b\\c&-d\esm\right) \rho^{-1}\left( \sm -a&b\\c&-d\esm\right)_{i,i}\ _1f_1(s, k; 2\pi in_{i,0}/(ac)) \bigg) \bigg]\\
  &&=- \frac{ (-1)^{k/2}}{2} (2\pi)^k (n_{i,0}+\kappa_i)^{k-1}\sum_{\substack{(c,d)\in\mathbb{Z}^2 \\ (c,d)=1, ac>0}}    c^{-k}
 \sum_{\nu=0}^n \binom{n}{\nu} \left( \frac ca \right)^{s} \left( \log\left( \frac ca\right)\right)^\nu 
 \\
\nonumber &&\quad\times  \frac{d^{n-\nu}}{ds^{n-\nu}}\bigg[\bigg(e^{2\pi i(n_{i,0}+\kappa_j)d/c} e^{\pi is}\chi^{-1}\left( \sm a&b\\c&d\esm\right)  \rho^{-1}\left( \sm a&b\\c&d\esm\right)_{i,i}\ _1f_1(s, k; -2\pi in_{i,0}/(ac))\\
\nonumber &&\quad \quad + e^{-2\pi i(n_{i,0}+\kappa_i)d/c} e^{-\pi is} \chi^{-1}\left( \sm -a&b\\c&-d\esm\right) \rho^{-1}\left( \sm -a&b\\c&-d\esm\right)_{i,i}\ _1f_1(s, k; 2\pi in_{i,0}/(ac)) \bigg) \bigg].
 \end{eqnarray}
 In the above equation, the derivative in the last two lines is equal to
 \begin{eqnarray*}
 &&\sum_{w=0}^{n-\nu} \binom{n-\nu}{w} \bigg\{ e^{2\pi i(n_{i,0}+\kappa_j)d/c}\chi^{-1}\left( \sm a&b\\c&d\esm\right)  \rho^{-1}\left( \sm a&b\\c&d\esm\right)_{i,i}\\
 &&\qquad \qquad  \qquad \qquad \times \frac{d^w}{ds^w} \big[e^{\pi is}\big] \frac{d^{n-\nu-w}}{ds^{n-\nu-w}} \big[\ _1f_1(s, k; -2\pi in_{i,0}/(ac))\big]\\
 && \qquad \qquad  \qquad+ e^{-2\pi i(n_{i,0}+\kappa_i)d/c} \chi^{-1}\left( \sm -a&b\\c&-d\esm\right) \rho^{-1}\left( \sm -a&b\\c&-d\esm\right)_{i,i}\\
 &&\qquad \qquad  \qquad \qquad \times\frac{d^w}{ds^w} \big[e^{-\pi is}\big] \frac{d^{n-\nu-w}}{ds^{n-\nu-w}} \big[\ _1f_1(s, k; 2\pi in_{i,0}/(ac))\big]\bigg\}\\
  &&\sum_{w=0}^{n-\nu} \binom{n-\nu}{w} \bigg\{ \left(\pi i\right)^w 
  e^{2\pi i(n_{i,0}+\kappa_j)d/c}e^{\pi is}\chi^{-1}\left( \sm a&b\\c&d\esm\right)  \rho^{-1}\left( \sm a&b\\c&d\esm\right)_{i,i}\\
 &&\qquad \qquad  \qquad \qquad \times  \frac{d^{n-\nu-w}}{ds^{n-\nu-w}} \big[\ _1f_1(s, k; -2\pi in_{i,0}/(ac))\big]\\
 && \qquad \qquad  \qquad+ \left(-\pi i\right)^w  e^{-2\pi i(n_{i,0}+\kappa_i)d/c} e^{-\pi is}\chi^{-1}\left( \sm -a&b\\c&-d\esm\right) \rho^{-1}\left( \sm -a&b\\c&-d\esm\right)_{i,i}\\
 &&\qquad \qquad  \qquad \qquad \times\frac{d^{n-\nu-w}}{ds^{n-\nu-w}} \big[\ _1f_1(s, k; 2\pi in_{i,0}/(ac))\big]\bigg\}.
 \end{eqnarray*}
 
By \cite[13.2.1]{AS}, for $\mathrm{Re}(\beta) > \mathrm{Re}(\alpha)>0$,  we have
\[
\ _1 f_1(\alpha, \beta; z) = \int_0^1 e^{zu} u^{\alpha-1} (1-u)^{\beta-\alpha-1} du.
\]
Therefore,  for any $n\in\mathbb{Z}_{\geq0}$, we obtain
\begin{eqnarray*}
&&\frac{d^n}{ds^n} \left[\ _1f_1\left( s, k; \pm \frac{2\pi in_{i,0}}{ac}\right)\right] \\
&&= \int_0^1 e^{\pm \frac{2\pi in_{i,0}}{ac} u} \frac{d^n}{ds^n} \left[ u^{s-1}(1-u)^{k-s-1}\right] du \\
&&= \int_0^1  e^{\pm \frac{2\pi in_{i,0}}{ac} u}  \left( \sum_{j=0}^n (-1)^{n-j} \binom{n}{j} (\log(u))^j (\log(1-u))^{n-j}\right) u^{s-1}(1-u)^{k-s-1}du.
\end{eqnarray*}
Since $\log(u)= o(u^{-\epsilon'})$ for any $\epsilon'>0$ as $u\to0$, we see that 
\[
\left| \frac{d^n}{ds^n} \left[\ _1f_1\left( s, k; \pm \frac{2\pi in_{i,0}}{ac}\right)\right] \right| \leq K_n,
\]
where $K_n$ is a constant depending only on $n$. 

Let $s = \frac k2 - \delta + it_0\ (\epsilon <\delta<\frac12)$. 
Then, the series in (\ref{thirdterm}) is
\begin{eqnarray*}
&&\ll \sum_{a=1}^\infty \sum_{c=1}^\infty  a^{-\frac k2 + \delta} c^{-\frac k2 - \delta} \bigg( 2\left| \log\left( \frac ca\right)\right|^n e^{\pi  |t_0|}K_0 \\
&&\qquad \qquad \qquad + 2e^{\pi  |t_0|}  \sum_{\nu=0}^{n-1} \binom{n}{\nu} \left| \log\left( \frac ca\right)\right|^\nu \sum_{w=0}^{n-\nu} \binom{n-\nu}{w} \left(\frac \pi 2\right)^w K_{n-\nu-w}\bigg).
\end{eqnarray*}
This can be estimated in terms of the Riemann zeta function and a positive constant factor $B(t_0, n)$ depending only on $t_0$ and $n$. 
If we divide the second term on the right-hand side of (\ref{derivative}) by  $(2\pi)^s  (n_{i,0}+\kappa_i)^{s-1} \Gamma(k-s)$, then the absolute value is
\begin{eqnarray*}
\ll \frac{(2\pi(n_{i,0}+\kappa_i))^{\frac k2 + \delta}}{\Gamma(\frac k2+\delta-it_0)} B(t_0, n)
\end{eqnarray*}
and this goes to $0$ as $k\to\infty$ uniformly in $\delta\in (\epsilon, \frac12)$ by Stirling's formula. 

In conclusion, if we divide both sides of (\ref{derivative}) by  $(2\pi)^s  (n_{i,0}+\kappa_i)^{s-1} \Gamma(k-s)$, then the right-hand side goes to zero as $k\to\infty$ but the absolute value of the left-hand side is 
\[
\gg |\log(\frac k2 + \delta - it_0)|^n
\]
as $k\to\infty$. 
This is a contradiction. 
\end{proof}

By using the functional equation of $L^*(f, s)$ ($f\in S_{k,\chi,\rho}$), we obtain the following corollary.

\begin{cor}
Let $k$, $\chi$, and $\{f_{k,1}, \ldots, f_{k,g_k}\}$  be as in Theorem \ref{main}.
Let $t_0\in \mathbb{R}, \epsilon>0$, and $n$ a positive integer. 
Then, for $k>C(t_0, \epsilon, n)$ and any $s = \sigma + it$ with $t = t_0,\ \frac{k-1}2 < \sigma < \frac k2 - \epsilon,\ \frac k2+\epsilon < \sigma< \frac{k+1}2$, there exists $f_{k,l}$ such that $\frac{d^n}{ds^n} L^*(f, s) \neq 0$.
\end{cor}

\section{The case of $\Gamma_0(N)$}

 Now, we consider the case of an elliptic modular form of integral weight on the congruence subgroup $\Gamma_0(N)$. By using Theorem \ref{main}, we can extend  a result in \cite{KSW} to the case of $\Gamma_0(N)$. To illustrate,
  let $N$ be a positive integer and let  $k$ be a positive even integer. 
  Let $\Gamma = \Gamma_0(N)$ and let $S_k(\Gamma)$ be the space of cusp forms of weight $k$ on $\Gamma$. 
  Let $\{\gamma_1, \ldots, \gamma_m\}$ be the set of representatives of $\Gamma \setminus \mathrm{SL}_2(\mathbb{Z})$ with $\gamma_1 = I$.
    For $f\in S_k(\Gamma)$, we define a vector-valued function $\tilde{f}:\mathbb{H}\to \mathbb{C}^m$ by $\tilde{f} = \sum_{j=1}^m f_j\mbf{e}_j$ and
    \[
    f_j = f|_k \gamma_j\ (1\leq j\leq m),
    \]
    where $(f|_k \sm a&b\\c&d\esm)(z) := (cz+d)^{-k} f(\gamma z)$.
    Then, $\tilde{f}$ is a vector-valued modular form of weight $k$ and the trivial multiplier system  with respect to $\rho$ on $\mathrm{SL}_2(\mathbb{Z})$, where $\rho$ is a certain $m$-dimensional unitary complex representation such that $\rho(\gamma)$ is a permutation matrix for each $\gamma\in \mathrm{SL}_2(\mathbb{Z})$ and is an identity matrix if $\gamma\in\Gamma$. 
    Then, the map $f\mapsto \tilde{f}$ induces an  isomorphism between $S_k(\Gamma)$ and $S_{k,\rho}$, where $S_{k,\rho}$ denotes the space of vector-valued cusp forms of weight $k$ and trivial multiplier system with respect to $\rho$ on $\mathrm{SL}_2(\mathbb{Z})$.
    
    For $\tilde{f}, \tilde{g}\in S_{k,\rho}$, we define a Petersson inner product by
        \begin{eqnarray*}
    (\tilde{f}, \tilde{g})& :=& \int_{\mathcal{F}} <\tilde{f}, \tilde{g}> y^k \frac{dxdy}{y^2}. 
    \end{eqnarray*}
    Note that  if $f, g\in S_{k}(\Gamma)$ such that $f$ and $g$ are orthogonal, then $\tilde{f}$ and $\tilde{g}$ is also orthogonal.

    \begin{cor}
    Let $k$ be a positive even integer with $k>2$.
    Let $N$ and $n$ be  positive integers and $\Gamma = \Gamma_0(N)$.
    Let $\{f_{k,1}, \ldots, f_{k, e_k}\}$ be an orthogonal basis of $S_k(\Gamma)$.
   Let $t_0\in\mathbb{R}, \epsilon>0$. 
    Then, there exists a constant $C(t_0, \epsilon, n)>0$ such that for $k>C(t_0, \epsilon,n )$ there exists a basis element $f_{k,l}\in S_{k}(\Gamma)$ satisfying
  \[
  \frac{d^n}{ds^n} L^*(\widetilde{f_{k,l}}, s) \neq 0
  \]
 at any point $s=\sigma + it_0$ with $$\frac{k-1}{2}<\sigma < \frac k2 - \epsilon \ \ and \ \ \frac k2 + \epsilon < \sigma < \frac{k+1}2.$$

    \end{cor}

\section{The case of Jacobi forms}

We now consider the case of Jacobi forms. 
Let $k$ be a positive even integer and $m$ be a positive integer. 
From now, we use the notation $\tau = u+iv\in\HH$ and $z = x+iy\in\CC$.
We review basic notions of Jacobi forms (for more details, see  \cite[Section 3.1]{CL} and \cite[Section 5]{EZ}).
Let $F$ be a complex-valued function on $\mathbb{H}\times \mathbb{C}$.
For $\gamma=\sm a&b\\c&d\esm \in\mathrm{SL}_2(\mathbb{Z}) , X = (\lambda, \mu)\in \mathbb{Z}^2$, we define
\[(F|_{k,m} \gamma)(\tau,z) := (c\tau+d)^{-k}e^{-2\pi im\frac{cz^2}{c\tau+d}}F(\gamma(\tau,z))\]
and
\[(F|_m X)(\tau,z) :=e^{2\pi i m (\lambda^2 \tau + 2\lambda z)}F(\tau,z+\lambda\tau+\mu),\]
where $\gamma(\tau,z) = (\frac{a\tau+b}{c\tau+d},\frac{z}{c\tau+d})$.

We give now the definition of a Jacobi form.

\begin{dfn}
A {\it{Jacobi form}} of weight $k$ and index $m$ on $\mathrm{SL}_2(\mathbb{Z})$  is a holomorphic function $F$ on $\mathbb{H}\times\mathbb{C}$ satisfying
\begin{enumerate}
\item[(1)] $F|_{k,m} \gamma =F$ for every $\gamma\in\mathrm{SL}_2(\mathbb{Z})$,
\item[(2)] $F|_m X = F$ for every $X\in \mathbb{Z}^2$,
\item[(3)] $F$ has the Fourier expansion of the form
\begin{equation} \label{Jacobifourier}
F(\tau,z) =
\sum_{\substack{l, r\in\mathbb{Z}\\ 4ml - r^2 \geq0}}a(l,r)e^{2\pi il\tau}e^{2\pi irz}.
\end{equation}
\end{enumerate}
\end{dfn}

We denote by $J_{k,m}$ the space of all Jacobi forms of weight $k$ and index $m$ on $\mathrm{SL}_2(\mathbb{Z})$. 
If a Jacobi form satisfies the  condition $a(l,r)\neq 0$ only if $4ml - r^2>0$, then it is called a Jacobi cusp form. 
We denote by $S_{k,m}$ the  space of all Jacobi cusp forms of weight $k$ and  index $m$ on $\mathrm{SL}_2(\mathbb{Z})$.

Let $F$ be a Jacobi cusp form $F\in S_{k,m}$ with its Fourier expansion (\ref{Jacobifourier}).
We define the partial $L$-functions of $F$ by
\[
L(F,j,s) := \sum_{\substack{n\in\mathbb{Z}\\  n+j^2 \equiv 0 \pmod{4m}}} \frac{a\left(\frac{n+j^2}{4m}, j\right)} {\left(\frac{n}{4m}\right)^{s}}
\]
for $1\leq j\leq 2m$.
Moreover, $F$ can  be written as
\begin{equation} \label{thetaexpansion}
F(\tau,z) = \sum_{1\leq j\leq 2m}F_j(\tau)\theta_{m,j}(\tau,z)
\end{equation}
with uniquely determined holomorphic functions $F_j:\mathbb{H}\to\mathbb{C}$ and 
functions in $\{F_j|\ 1\leq j\leq 2m\}$ have the Fourier expansions 
\[
F_j(\tau) =\sum_{\substack{n\geq0\\ n+j^2 \equiv 0 \pmod{4m}}} a\left(\frac{n+j^2}{4m}, j\right)e^{2\pi in\tau/(4m)},
\]
where the theta series $\theta_{m, j}$ is defined by
\[
\theta_{m, j}(\tau,z) := \sum_{\substack{r\in\mathbb{Z}\\ r\equiv j \pmod{2m}}} e^{2\pi ir^2\tau/(4m)} e^{2\pi irz}
\]
for $1\leq j\leq 2m$.

We write $\Mp_2(\mathbb{R})$ for the metaplectic group.
The elements of $\Mp_2(\mathbb{R})$ are pairs $(\gamma,\phi(\tau))$, 
where $\gamma = \sm a&b\\c&d\esm\in\SL_2(\mathbb{R})$, and $\phi$ denotes a holomorphic function on $\HH$ with $\phi(\tau)^2= c\tau+d$.
The map
\[\sm a&b\\c&d\esm  \mapsto \widetilde{\sm a&b\\c&d\esm} = (\sm a&b\\c&d\esm, \sqrt{c\tau+d})\]
defines a locally isomorphic embedding of $\SL_2(\mathbb{R})$ into $\Mp_2(\mathbb{R})$. Let $\Mp_2(\mathbb{Z})$ be the inverse image of $\SL_2(\mathbb{Z})$ under the covering map $\Mp_2(\mathbb{R})\to\SL_2(\mathbb{R})$. It is well-known that $\Mp_2(\mathbb{Z})$ is generated by $\widetilde{T}$ and $\widetilde{S}$. 
 We define a  $2m$-dimensional unitary complex representation $\widetilde{\rho}_m$ of $\Mp_2(\mathbb{Z})$ by 
 \[
 \widetilde{\rho}_m(\widetilde{T})  \mbf{e}_j = e^{-2\pi ij^2/(4m)}\mbf{e}_j
 \]
 and
 \[
  \widetilde{\rho}_m(\widetilde{S})  \mbf{e}_j = \frac{i^{\frac 12}}{\sqrt{2m}}\sum_{j'=1}^{2m} e^{2\pi i jj'/(2m)}\mbf{e}_{j'},
 \]
 Let $\chi$ be a multiplier system of weight $\frac12$ on $\mathrm{SL}_2(\mathbb{Z})$.
 We define a map $\rho_m : \SL_2(\mathbb{Z}) \to \GL_{2m}(\mathbb{C})$ by
\[\rho_m(\gamma) = \chi(\gamma)\widetilde{\rho}_m(\widetilde{\gamma})\]
for $\gamma\in\SL_2(\mathbb{Z})$.
The map $\rho_m$ gives a $2m$-dimensional unitary representation of $\SL_2(\mathbb{Z})$.

Let $\{\mbf{e}_1, \ldots, \mbf{e}_{2m}\}$ denote the standard basis of $\mathbb{C}^{2m}$. 
For $F\in S_{k,m}$, we define a vector-valued function $\tilde{F}:\mathbb{H}\to \mathbb{C}^{2m}$ by $\tilde{F} = \sum_{j=1}^{2m} F_j\mbf{e}_j$, where $F_j$ is defined by the theta expansion in (\ref{thetaexpansion}).
Then, the map $F\mapsto  \tilde{F}$ induces an  isomorphism between $S_{k,m}$ and  $S_{k-\frac12, \bar{\chi}, \rho_m}$.

Let $L^*(F,j,s) := \frac{\Gamma(s)}{(2\pi)^s} L(F,j,s)$.
Then, we have the following corollary.

 \begin{cor}
   Let $k$ be a positive even integer with $k>2$. 
   Let $m$ and $n$ be positive integers. 
  Let $\{F_{k,m,1}, \ldots, F_{k,m,d}\}$ be an orthogonal basis of $S_{k,m}$.
  Let $t_0\in\mathbb{R}$ and $\epsilon>0$.
  \begin{enumerate}
     
  \item
  Let $j$ be a positive integer with $1\leq j\leq 2m$.
  Then, there exists a constant $C(t_0, \epsilon,j,n)>0$ such that for any $k>C(t_0, \epsilon, j,n)$,  and any $s=\sigma + it_0$ with $$\frac{2k-3}{4}<\sigma < \frac {2k-1}{4} - \epsilon,$$ there exists a basis element $F_{k,m,l}\in S_{k,m}$ such that
  \[
   \frac{d^n}{ds^n} L^*(F_{k,m,l},j,s)\neq 0.
  \]

   \item There exists a constant $C(t_0, \epsilon,n)>0$ such that for any $k>C(t_0, \epsilon,n)$, and any $s=\sigma + it_0$ with $$\frac{2k-3}{4}<\sigma < \frac {2k-1}{4} - \epsilon \ \ and \ \ \frac {2k-1}{4} + \epsilon < \sigma < \frac{2k+1}4,$$ there exist a basis element $F_{k,m,l}\in S_{k,m}$ and $j\in\{1,\ldots, 2m\}$ such that
  \[
  \frac{d^n}{ds^n} L^*(F_{k,m,l},j,s) \neq 0.
  \]

  \end{enumerate}
  
  \end{cor}

  \begin{rmk}
Note that $\rho_m(-I)$ is not equal to the identity matrix in $\mathrm{GL}_{2m}(\mathbb{C})$. 
  Instead, we have
  \[
  \rho_m(-I) \mbf{e}_j = i \mbf{e}_{2m-j}.
  \]
  By a similar argument, we prove the same result as in Theorem \ref{main} for the representation $\rho_m$.
   \end{rmk}

   \section{The case of Kohnen plus space}
   Let $k$ be a positive even integer. 
By \cite[Theorem 5.4]{EZ}, there is an isomorphism $\phi$ between  $S_{k,1}$ and $S^+_{k-\frac12}$, where $S^+_{k-\frac12}$ denotes the space of cusp forms in the plus space of weight $k-\frac12$ on $\Gamma_0(4)$.
Moreover, this isomorphism is compatible with the Petersson scalar products.

Let $f$ be a cusp form in $S^+_{k-\frac12}$ with Fourier expansion $f(\tau) = \sum_{\substack{n>0\\ n\equiv 0,3 \pmod{4}}} c(n)e^{2\pi in\tau}$.
Then, the $L$-function of $f$ is defined by
\[
L(f,s) := \sum_{\substack{n>0\\ n\equiv 0,3 \pmod{4}}} \frac{c(n)}{n^s}.
\]
For $1\leq j\leq 2$, let $c_j$ be defined by
\[
c_j(n) :=
\begin{cases}
c(n) & \text{if $n\equiv -j^2 \pmod{4}$},\\
0 & \text{otherwise}.
\end{cases}
\]
Then, $c(n) = c_1(n) + c_2(n)$ for all $n$.
With this, we consider partial sums of $L(f,s)$ by
\[
L(f,j,s) :=  \sum_{\substack{n>0\\ n\equiv 0,3 \pmod{4}}} \frac{c_j(n)}{n^s}
\]
for $1\leq j\leq 2$.

Suppose that $F$ is a Jacobi cusp form in $S_{k,1}$.
By the theta expansion in (\ref{thetaexpansion}), we have a corresponding vector-valued modular form $(F_1(\tau), F_2(\tau))$. 
Then, the isomorphism $\phi$ from  $S_{k,1}$  to $S^+_{k-\frac12}$ is given by 
\[
\phi(F) = \sum_{j=1}^2 F_j(4\tau).
\]
 From this, we see that
\[
L(f,j,s) = \frac{1}{4^s} L(F,j,s). 
\]

We have the following corollary regarding the partial sums of $L(f,s)$ for $f\in S^+_{k-\frac12}$.

\begin{cor} \label{plus}
   Let $k$ be a positive even integer with $k>2$. 
   Let  $n$ be a positive integer. 
  Let $\{f_{k-\frac12,1}, \ldots, f_{k-\frac12,d}\}$ be an orthogonal basis of $S^+_{k-\frac12}$.
  Let $t_0\in\mathbb{R}$ and $\epsilon>0$.
  \begin{enumerate}
     
  \item
  Let $j$ be a positive integer with $1\leq j\leq 2$.
  Then, there exists a constant $C(t_0, \epsilon,j,n)>0$ such that for any $k>C(t_0, \epsilon,j,n)$,  and any $s=\sigma + it_0$ with $$\frac{2k-3}{4}<\sigma < \frac {2k-1}{4} - \epsilon,$$ there exists a basis element $f_{k-\frac12,l}\in S^+_{k-\frac12}$ such that
  \[
   \frac{d^n}{ds^n} \left[ 4^s L^*(f_{k,m,l},j,s) \right]\neq 0.
  \]

   \item There exists a constant $C(t_0, \epsilon,n)>0$ such that for any $k>C(t_0, \epsilon,n)$, and any $s=\sigma + it_0$ with $$\frac{2k-3}{4}<\sigma < \frac {2k-1}{4} - \epsilon \ \ and \ \ \frac {2k-1}{4} + \epsilon < \sigma < \frac{2k+1}4,$$ there exist a basis element $f_{k-\frac12,l}\in S^+_{k-\frac12}$ and $j\in\{1,\ldots, 2\}$ such that
  \[
  \frac{d^n}{ds^n} \left[ 4^s L^*(F_{k,m,l},j,s) \right] \neq 0.
  \]

  \end{enumerate}
  
  \end{cor}


    \end{document}